\documentclass[a4paper,11pt,pdf]{amsart}

\usepackage{enumerate, amsfonts,  wasysym,
graphics, graphicx, xcolor, frcursive,bbm, hyperref}

\numberwithin{equation}{section}
\newtheorem{thm}[equation]{Theorem}

\newtheorem{prop}[equation]{Proposition}
\newtheorem{proposition}[equation]{Proposition}
\newtheorem{lem}[equation]{Lemma}
\newtheorem{cor}[equation]{Corollary}

\newtheorem{conjecture}[equation]{Conjecture}

\newtheorem{condition}[equation]{Condition}

\theoremstyle{definition}

\newtheorem{definition}[equation]{Definition}
\newtheorem{rmq}[equation]{Remark}

\newtheorem{exple}[equation]{Example}

\newcommand\Red{{\mathrm{Red}_T}}

\newcommand\lt{{\ell_T}}
\newcommand\rk{{\mathrm{rank}}}
\newcommand\Ref{{\mathrm{Ref}}}

\usepackage{lipsum}

\title[On cycle decompositions in Coxeter groups]{On cycle decompositions in Coxeter groups}

\author{Thomas Gobet}

\thanks{The author is funded by the ANR Geolie ANR-15-CE40-0012}

\address{Institut Elie Cartan de Lorraine, Universit\'e de Lorraine, B.P. 70239, F-54506 Vandoeuvre-l\`es-Nancy Cedex.}

\begin{document}

\begin{abstract}
The aim of this note is to show that the cycle decomposition of elements of the symmetric group admits a quite natural formulation in the framework of dual Coxeter theory, allowing a generalization of it to the family of so-called \textit{parabolic quasi-Coxeter elements} of Coxeter groups (in the symmetric group every element is a parabolic quasi-Coxeter element). We show that such an element admits an analogue of the cycle decomposition. Elements which are not in this family still admit a generalized cycle decomposition, but it is not unique in general.
\end{abstract}

\maketitle

\section{Introduction}

The cycle decomposition in the symmetric group is a powerful combinatorial tool to study properties of permutations. On the other hand, the symmetric groups can be realized as Coxeter groups. It is easy for example to determine the order of an element from its cycle decomposition, hence even if we prefer to view the symmetric groups as Coxeter groups it is sometimes useful to represent their elements as permutations and make use of their unique cycle decomposition, rather than using Coxeter theoretic representations of the elements as words in the simple generating set.

It therefore appears as natural to wonder whether the cycle decomposition admits a natural generalization to Coxeter groups. However, when trying to define cycle decompositions in the symmetric group purely in terms of the classical Coxeter theoretic data, one rapidly sees an obstruction towards such a generalization: considering a Coxeter system $(W,S)$ of type $A_n$ (with $W$ identified with $\mathfrak{S}_{n+1}$ and $S$ with the set of simple transpositions) and an element $w\in W$ with cycle decomposition $w=c_1 c_2 \cdots c_k$, the Coxeter length $\ell_S(w)$ of $w$ is not equal in general to the sums of the lengths of the various $c_i$. 

However, replacing the generating simple set $S$ by the set $T$ of all transpositions and the classical length by the length function $\ell_T$ on $W$ with respect to $T$, one has that $\ell_T(w)=\sum_{i=1}^k \ell_T(c_i).$
The set of transpositions forms a single conjugacy class. From a Coxeter theoretic point of view, it is the set of \textit{reflections} of $W$, i.e., the set of $W$-conjugates of the elements of $S$. In particular the reflection length function $\ell_T$ can be defined for an abritrary Coxeter group. There are deep motivations for the study of a (finite) reflection group as a group generated by the set $T$ of all its reflections instead of just the set $S$ of reflections through the walls of a chamber. This approach, nowadays called the \textit{dual} approach, has been a very active field of research in the last fifteen years (see for instance \cite{Dual}, \cite{BWnonc}, \cite{Arm}, \cite{IT}, \cite{Read}).

The above basic observation on the reflection length of a permutation indicates that the cycle decomposition has something which is \textit{dual} in essence. Each cycle can be thought of as a Coxeter element in an irreducible parabolic subgroup of $W$. From the type $A_n$ picture, it therefore appears as natural to generalize a cycle decomposition as a decomposition of an element $w$ of Coxeter system $(W,S)$ into a product of Coxeter elements in irreducible reflection subgroups of $W$, which pairwise commute, and such that the sum of their reflection lengths equals the reflection length of $w$. However, even for finite $W$ there are in general elements failing to admit such a decomposition. In order to make it work, one has to relax the definition of Coxeter element to that of a \textit{quasi-Coxeter element} (in type $A_n$ both are equivalent). Namely, given $w\in W$ and denoting by $\Red(w)$ the set of $T$-reduced expressions of $w$, that is, minimal length expressions of $w$ as product of reflections, we say that $w\in W$ is a \emph{parabolic quasi-Coxeter element} if it satisfies the following condition: 

\begin{condition}\label{con:1}
There exists $(t_1,t_2,\dots, t_k)\in\Red(w)$ such that $W':=\langle
t_1,t_2,\dots, t_k\rangle$ is a parabolic subgroup of $W$.
\end{condition}
If the parabolic subgroup is the whole group $W$ then we just call $w$ a \textit{quasi-Coxeter element}. In type $D_n$ for instance, there are quasi-Coxeter elements which fail to be Coxeter elements (that is, with no $T$-reduced expression yielding a simple system for $W$; see \cite{BGRW}). 

The parabolic subgroup in the above Condition is unique in the sense that if another reduced expression $(q_1, q_2, \dots, q_k)\in \Red(w)$ generates a parabolic subgroup $W''$ of $W$, then $W'=W''$ (see Lemma~\ref{lem:parabw}). We therefore denote $W'$ by $P(w)$. In this situation it is easy to derive


\begin{prop}[Generalized cycle decomposition]\label{main}
Let $(W,S)$ be a Coxeter system. Let $w\in W$ satisfying Condition~\ref{con:1}. There exists a (unique up to the order of the factors) decomposition $w=x_1 x_2\cdots x_m$, $x_i\in W$ such that
\begin{enumerate}
\item $x_i x_j= x_j x_i$ for all $i,j=1,\dots, m$,
\item $\lt(w)=\lt(x_1)+\lt(x_2)+\cdots+\lt(x_m)$,
\item Each $x_i$ admits a $T$-reduced expression generating an
irreducible parabolic subgroup $W_i$ of $W$ and
$$P(w)=W_1\times
W_2\times\cdots\times W_m.$$
\end{enumerate}
\end{prop}
This statement is not entirely satisfying in the sense that we would like to state the maximality condition given in point (3) only in terms of the factors $x_i$ and not in terms of the parabolic subgroups. More precisely, we expect the $x_i$'s to be \textit{indecomposable}, that is, to admit no nontrivial decomposition of the form $u_i v_i$ with $u_iv_i=v_i u_i$ and $\ell_T(x_i)=\ell_T(u_i)+\ell_T(v_i)$. For finite groups at least this can be achieved (see Proposition~\ref{lemquasi}) yielding

\begin{thm}[Generalized cycle decomposition in finite Coxeter groups]\label{mainfinite}
Let $(W,S)$ be a finite Coxeter system. Let $w\in W$ satisfying Condition~\ref{con:1}. There exists a (unique up to the order of the factors) decomposition $w=x_1 x_2\cdots x_m$, $x_i\in W$ such that 
\begin{enumerate}
\item $x_i x_j= x_j x_i$ for all $i,j=1,\dots, m$,
\item $\lt(w)=\lt(x_1)+\lt(x_2)+\cdots+\lt(x_m)$,
\item Each $x_i$ is indecomposable. 
\end{enumerate}
\end{thm}

Theorem \ref{mainfinite} is the analogue of the cycle decomposition for parabolic quasi-Coxeter elements. In general there are elements in $W$ failing to be parabolic quasi-Coxeter elements, but the advantage of this definition is that such an element is always a quasi-Coxeter element in a reflection subgroup (but this subgroup is never unique when $W$ is finite: in that case by Corollary \ref{cor:corresp} below these reflection subgroups are in bijection with the number of Hurwitz orbits on $\Red(w)$; by \cite[Theorem 1.1]{BGRW} this number is one precisely when $w$ satisfies Condition~\ref{con:1}).   


\section{Coxeter groups and their parabolic subgroups}\label{sec:cox}

Let $(W, S)$ be a (not necessarily finite) Coxeter system of rank $n=|S|$. We assume the reader to be familiar with the general theory of Coxeter groups and refer to \cite{Bou} or \cite{Humph} for basics on the topic. Let $T=\bigcup_{w\in W} w S w^{-1}$ be
the set of \emph{reflections} of $W$. Let $\lt:W\rightarrow\mathbb{Z}_{\geq 0}$ be the \emph{reflection length}, that is, for $w\in W$ the integer $\ell_T(w)$ is the smallest possible length of an expression of $w$ as product of reflections. We write $\leq_T$ for the \emph{absolute order} on $W$, that is, for $u, v\in W$ we set
$$u \leq_T v~\Leftrightarrow~\lt(u)+\lt(u^{-1}v)=\lt(v).$$

Given $w\in W$, we denote by $\Red(w)$ the set of \emph{$T$-reduced expressions} of $w$, that is, the set of minimal length expressions for
$w$ as products of reflections.

\begin{definition}\label{defn:parabolic}
A subgroup $W'\subseteq W$ is \emph{parabolic} if it exists a subset $S'=\{ r_1, r_2, \dots, r_n\}\subseteq T$ and $m\leq n$ such that $(W, S')$ is a Coxeter system and $W'=\langle r_1, r_2, \dots, r_m\rangle$. 
\end{definition}

The above definition, which is borrowed from \cite{BDSW}, is more general than the usual definition of parabolic subgroups as conjugates of subgroups generated by subsets of $S$. In~\cite[4.4 and 4.6]{BGRW} it is shown that the above definition is equivalent to the classical one for finite and irreducible $2$-spherical Coxeter groups. The example below shows that the two definitions are not equivalent in general: 

\begin{exple}
Let $W$ be a universal Coxeter group on three generators $S=\{s, t, u\}$, that is, with no relation between distinct generators. Then $S':=\{s, t, tut\}\subseteq T$ is a simple system for $W$, hence $X:=\langle s, tut\rangle$ is parabolic. However, using the fact that elements of $W$ have a unique $S$-reduced expression (because $W$ is universal) it is easy to check that $X$ is not conjugate to any of the three rank $2$ standard parabolic subgroups of $W$. 

\end{exple}

Parabolic subgroups
provide a family of \emph{reflection subgroups} of $W$, that is, subgroups
generated by reflections. Any reflection subgroup $W'\subseteq W$ comes equipped with a
canonical structure of Coxeter group (see~\cite{Dyer}), in particular it has a canonical set $S'$ of Coxeter generators. Moreover by \cite[Corollary 3.11 (ii)]{Dyer} the set $\mathrm{Ref}(W')$ of $W'$-conjugates of $S'$ (the \textit{reflections} of $W'$) coincide with $W'\cap T$. The \emph{rank}
$\rk(W')$ of $W'$ is defined to be $|S'|$.

The following result will be useful: 

\begin{thm}[{\cite[Theorem 1.4]{BDSW}}]\label{thm:parabreduit}
Let $W'\subseteq W$ be a parabolic subgroup. Let $w\in W'$. Then $\mathrm{Red}_{T'}(w)=\Red(w),$
where $T'=W'\cap T$ is the set of reflections of $W'$. 
\end{thm}

In this context it seems natural to us to conjecture the following: 

\begin{conjecture}\label{conj:parabeng}
Let $w\in W$. Assume that there is $(t_1, t_2,\dots, t_k)\in\Red(w)$ such that $W':=\langle t_1, t_2, \dots, t_k\rangle$ is parabolic. Then for any $(q_1, q_2,\dots, q_k)\in\Red(w)$ we have $W'=\langle q_1, q_2, \dots, q_k\rangle$. 
\end{conjecture}

Note that

\begin{thm}[{\cite{BGRW}, \cite{BDSW}}]
Conjecture~\ref{conj:parabeng} holds in the following cases: 
\begin{enumerate}
\item When $W$ is finite,
\item When $w$ is a \emph{parabolic Coxeter element} in $W$, that is, if it exists $S'=\{r_1, r_2, \dots, r_n\}\subseteq T$ and $m\leq n$ such that $w= r_1 r_2\cdots r_m$ and $S'$ is a simple system for $W$.  
\end{enumerate}
\end{thm}

\begin{proof}
Conjecture~\ref{conj:parabeng} for finite $W$ is an immediate consequence of \cite[Theorem 1.1]{BGRW}: there it is shown that an element satisfies Condition~\ref{con:1} if and only if the Hurwitz action (see Section~\ref{hurwitz} for the definition) is transitive on $\Red(w)$; but this action leaves the subgroup generated by the reflections in a $T$-reduced expression invariant. It also holds for parabolic Coxeter elements since in that case the Hurwitz action is also transitive on $\Red(w)$ by \cite[Theorem 1.3]{BDSW}.  
\end{proof}

\begin{lem}\label{lem:parabw}
Let $w\in W$. Assume that $(t_1, t_2, \cdots, t_k), (q_1, q_2, \dots, q_k)\in\Red(w)$ are such that both $W':=\langle t_1, t_2, \dots, t_k\rangle$ and $W'':=\langle q_1, q_2, \dots, q_k\rangle$ are parabolic. Then $W'=W''$. 
\end{lem}
\begin{proof}
Since $W'$ is parabolic, by Theorem~\ref{thm:parabreduit} we have $q_i\in W'$ for all $i$, hence $W''\subseteq W'$. Reversing the roles of $W'$ and $W''$ we get $W'\subseteq W''$. 
\end{proof}
The lemma above allows the following definition

\begin{definition}
Let $w\in W$ satisfying Condition~\ref{con:1}. We denote by $P(w)$ the parabolic subgroup of $W$ generated by any $T$-reduced decomposition of $w$ generating a parabolic subgroup. This is well-defined by Lemma~\ref{lem:parabw}.
\end{definition}

It follows immediately from Theorem~\ref{thm:parabreduit} that any parabolic subgroup $P$ containing $w$ must contain $P(w)$. We call $P(w)$ the \textit{parabolic closure} of $w$. 

\begin{lem}\label{lem:product}
Let $W'\subseteq W$ be a finitely generated reflection subgroup. Then there is a unique (up to the order of the factors) decomposition $W'= W_1\times W_2\times \cdots \times W_k$ where $W_1, W_2,\dots, W_k$ are irreducible reflection subgroups of $W'$ and $\mathrm{Ref}(W)=\dot{\bigcup}_{i=1}^k \mathrm{Ref}(W_i)$. 
\end{lem}
\begin{proof}
Let $S'$ be the canonical set of Coxeter generators of $W'$ (see \cite{Dyer}). By \cite[Corollary 3.11]{Dyer}, $S'$ is finite and $\bigcup_{w\in W'} w S' w^{-1}= W'\cap T$. If $\Gamma_1, \Gamma_2, \dots, \Gamma_k$ are the irreducible components of the Coxeter graph of $(W', S')$, then $W'= W_1\times W_2\times\cdots \times W_k$ (where $W_i$ is generated by the nodes of $\Gamma_i$) and each $W_i$ is an irreducible reflection subgroup. Now if there is another decomposition $W= W'_1 \times W'_2\times\cdots \times W'_\ell$, then all the reflections in $W_i'$ must be included in $W_j$ for some $j$ (otherwise irreducibility is not satisfied) and vice-versa, implying uniqueness of the decomposition. 
\end{proof}

\section{Generalized cycle decompositions}

\subsection{Proof of Proposition~\ref{main}}\label{sec:main}

\begin{proof}[Proof of Proposition~\ref{main}]
Let $(t_1,t_2,\dots, t_k)\in \Red(w)$ such that $P:=\langle t_1,
t_2,\dots, t_k\rangle$ is parabolic. By Lemma~\ref{lem:parabw} we have
$P=P(w)$. It follows from the definition of a parabolic subgroup that there is a (unique up to the order of the factors)
factorization $$P=W_1\times W_2\times\cdots \times W_m$$ where the $W_i$'s
are irreducible parabolic subgroups. Moreover we have $k=\rk (P)=\sum_{i=1}^m
\rk(W_i)$ and $\Ref(P)=\dot{\bigcup}_{i=1}^m \Ref(W_i)$.
It follows that for each $j=1,\dots, k$, there exists $j'\in\{1, \dots,
m\}$ such that $t_j\in W_{j'}$. This implies that we can transform the $T$-reduced expression $(t_1, t_2,\dots,
t_k)$ by a sequence of commutations of adjacent letters into a $T$-reduced expression $(q_1, \dots, q_{\ell_1}, q_{\ell_1+1}, \dots,
q_{\ell_2}, \dots, q_{\ell_{m-1}}, \dots, q_k)$
of $w$ where $$\{ q_1, \dots, q_{\ell_1-1}\}\subseteq W_1, \dots, \{
q_{\ell_{m-1}}, \dots, q_{k}\}\subseteq W_m.$$ Note that since the set $\{ t_1,\dots, t_k\}$ generates $P$ we must have $$\langle q_1, \dots, q_{\ell_1-1}\rangle= W_1, \dots, \langle
q_{\ell_{m-1}}, \dots, q_{k}\rangle= W_m.$$ Setting $\ell_0=1$ and
$\ell_m=k+1$ we define $x_i:= q_{\ell_{i-1}}
q_{\ell_{i-1}+1}\cdots q_{\ell_i-1}$ for all $i=1,\dots, m$. Note that since the $W_i$'s are irreducible and parabolic we get $(3)$. As $(q_1,\dots q_k)\in\Red(w)$ is given by concatenating $T$-reduced expressions of
the $x_i$'s we have $\lt(w)=\lt(x_1)+\lt(x_2)+\dots+\lt(x_m)$ which
shows $(2)$. Since $x_i\in W_i$ for all $i$ and $P=W_1\times
W_2\times\cdots\times W_m$ we have $x_ix_j=x_jx_i$ for all $i,j=1,\dots,
m$, which shows $(1)$. 


It remains to show that the decomposition is unique up to the order of the factors. Hence assume that $w=y_1 y_2\cdots y_{m'}$ is another
decomposition of $w$ satisfying the three conditions of Proposition~\ref{main}. By the third condition each of the $y_i$'s
has a reduced expression generating an irreducible parabolic subgroup
$W_i'=P(y_i)$ of $W$ and concatenating them yields a reduced expression generating $P(w)$. By uniqueness of the decomposition $P(w)=W_1\times \cdots\times
W_m$ we must have $m=m'$ and there must exist a
permutation $\pi\in\mathfrak{S}_m$ such that $W_i'= W_{\pi(i)}$ for all $i$. Up
to reordering, we can therefore assume that $W_i'=W_i$ for all $i$. Since $x=x_1 x_2\cdots x_m=y_1 y_2\cdots y_m$ it follows by uniqueness of the decomposition of $w$ as element of the direct product $W_1\times W_2\times\cdots\times W_m$ that $x_i=y_i$ for all $i=1,\dots, m$. 


\end{proof}

\begin{rmq}
In type $A_n$ we recover the classical cycle decomposition. In that case each element $w$ satisfies Condition~\ref{con:1}. The second condition in Proposition~\ref{main} follows from \cite[Lemma 2.2]{Brady}.
\end{rmq}

\subsection{Finite Coxeter groups and their parabolic subgroups}
This section is devoted to recalling well-known facts on parabolic
subgroups of finite Coxeter groups and their connexion with finite root systems. Most of what we present
here is covered in~\cite{Bou}, though often in different notations. From now on we always assume $(W,S)$ to be finite.

Let $(W,S)$ be finite, of rank $n$. Let $\Phi$ be a root
system for $(W,S)$ in an $n$-dimensional Euclidean space $V$ with inner product $(\cdot,\cdot)$. Let $\Phi^+\subseteq\Phi$ be a positive
system. Recall that there is a one-to-one
correspondence between $T$ and $\Phi^+$ which we denote by
$t\mapsto\alpha_t$. Let $w\in W$ and let $$V^w:=\{v\in V~|~w(v)=v\}$$ be the subspace of $V$ consisting of the fixed
points under the action of $w$. The following well-known result is due to Carter~\cite[Lemma 3]{Carter}

\begin{lem}[Carter's Lemma]\label{carter}
Let $\alpha_{t_1}, \alpha_{t_2}, \dots, \alpha_{t_k}\in\Phi^+$. Then
$\{\alpha_{t_1}, \alpha_{t_2}, \dots, \alpha_{t_k}\}$ is linearly
independent if and only if $\lt(t_1 t_2\cdots t_k)=k$. In that case one
has $\dim V^{t_1 t_2\cdots t_k}=n-k$ and $W':=\langle t_1, t_2, \dots, t_k\rangle$ is a reflection subgroup of rank $k$ of $W$.
\end{lem}




The following will be useful (see~\cite[Lemma 1.2.1 (i)]{Dual})

\begin{lem}\label{ref}
Let $x\in W$, $t\in T$. Then $t\leq_T x\Leftrightarrow V^x\subseteq V^t.$
\end{lem}

Given $w\in W$, there is an orthogonal decomposition $V=V^w\oplus\mathsf{Mov}(w)$ with respect to $(\cdot,\cdot)$, where $\mathsf{Mov}(w):=\mathrm{im}(w-1)$ (see for instance \cite[Section 2.4]{Arm}).

Recall that for finite Coxeter groups, Definition~\ref{defn:parabolic} is equivalent to the classical definition, that is, $W'\subseteq W$ is parabolic if and only if $W'$ is a conjugate of a standard parabolic subgroup $W_I$ of $W$, where for $I\subseteq S$ we write $W_I:=\langle s~|~s\in I\rangle$. There is the following result, characterizing parabolic subgroups of finite Coxeter groups as
centralizers of subspaces of $V$ (which is a Corollary of~\cite[3.3, Proposition 1]{Bou}).

\begin{proposition}\label{prop:bou}
Let $P\subseteq W$ be a parabolic subgroup. Then
$$P=\mathrm{Fix}(E):=\{x\in W~|~x(v)=v,~\forall v\in E\}$$ for some subspace
$E\subseteq V$. Conversely, given any subspace $E\subseteq V$, the
subgroup $\mathrm{Fix}(E)$ is a parabolic subgroup of
$W$.
\end{proposition}
In fact, the subspaces $E$ in Proposition~\ref{prop:bou} can be chosen to be
intersections of reflection hyperplanes. Given linearly independent reflection hyperplanes $V_{t_1}, \dots, V_{t_k}$ where $t_i\in T, i=1,\dots,k$ and setting $w:= t_1 t_2\cdots
t_k$, it follows from Carter's Lemma (Lemma~\ref{carter}) that
$V^w=\bigcap_{i=1}^k V_{t_i},$
and $\dim (V^w)=n-\lt(w)$ (see also \cite[Lemma 1.2.1 (ii)]{Dual}). It follows from the discussion above that
$P(w)=\mathrm{Fix} (V^w)$ and $\rk(P(w))=\lt(w)$. Every parabolic subgroup is
the parabolic closure of some (in general not uniquely determined) element.

\subsection{Hurwitz action and proof of Theorem~\ref{mainfinite}}\label{hurwitz}

We now give a few properties of elements of finite Coxeter groups satisfying Condition~\ref{con:1} before proving Theorem~\ref{mainfinite}. These elements were introduced in \cite{BGRW} and called \textit{quasi-Coxeter elements}. 


Recall that for each $w\in W$ with $\ell_T(w)=k$, there is an action of the $k$-strand Artin braid group $\mathcal{B}_k$ on $\Red(w)$ called the \emph{Hurwitz action}, defined as follows. The Artin generator $\sigma_i\in \mathcal{B}_k$ acts by $$\sigma_i\cdot (t_1, \dots, t_{i-1}, t_i, t_{i+1}, t_{i+2}, \dots, t_k)=(t_1, \dots, t_{i-1}, t_i t_{i+1} t_i, t_{i}, t_{i+2}, \dots, t_k).$$

In \cite[Theorem 1.1]{BGRW}, it is shown that this action is transitive if and only if $w$ satisfies Condition~\ref{con:1}. Since the reflection subgroup generated by the reflections from a reduced expression is invariant under a Hurwitz move as above, this means that either every $T$-reduced expression generates a parabolic subgroup (which by Lemma~\ref{lem:parabw} is nothing but $P(w)$) or no reduced expression does.  

With the following Proposition (which requires the above mentioned result from \cite{BGRW} for which we only have a case-by-case proof) it will be easy to derive a proof of Theorem~\ref{mainfinite} 

\begin{prop}\label{lemquasi}
Let $(W,S)$ be a finite irreducible Coxeter system. Let $x$ be a quasi-Coxeter element in $W$. Then there is no nontrivial decomposition $x=uv=vu$ such that $u,v\in W$ and $$\ell_T(x)=\ell_T(u)+\ell_T(v).$$
\end{prop}

\begin{proof}
Assume that there is such a decomposition $x=uv$. For $y\in\{ u,v\}$ define $$W_y:=\langle t\in T~|~t\leq_T y\rangle.$$
We claim that in that case we have $W=W_u\times W_v$ and $$\mathrm{Ref}(W)=\mathrm{Ref}(W_u)\dot{\cup} \mathrm{Ref}(W_v),$$ contradicting the irreducibility of $W$. 

Firstly we show that $\mathsf{Mov}(v)\subseteq V^u$. We have $V^{u}\cap V^{v}\subseteq V^{uv}$, hence $\mathsf{Mov}(uv)\subseteq \mathsf{Mov}(u)+\mathsf{Mov}(v)$. Since moreover $\ell_T(x)=\ell_T(uv)=\ell_T(u)+\ell_T(v)$ by Carter's Lemma we have $\mathsf{Mov}(u)\cap\mathsf{Mov}(v)=0$. Let $a\in\mathsf{Mov}(v)$. Then since $uv=vu$ we have $u(a)\in\mathsf{Mov}(v)$, hence $u(a)-a\in\mathsf{Mov}(v)\cap\mathsf{Mov}(u)=0$. Hence $a\in V^u$, which shows the claimed inclusion. 

Now if $t\in T$ is such that $t\leq_T u$, then by Lemma~\ref{ref} we have that $\mathsf{Mov}(v)\subseteq V^u\subseteq V^t$ which implies that $\alpha_t\in V^v$. Using~\ref{ref} again, we deduce that $t$ commutes with any reflection $t'\in T$ such that $t'\leq_T v$.

Concatenating a $T$-reduced expression $t_1 t_2\cdots t_k$ of $u$ with a $T$-reduced expression $t_{k+1} t_{k+2}\cdots t_n$ of $v$ we get a $T$-reduced expression $t_1 t_2\cdots t_n$ of $x$. By the discussion above we have $tt'=t't$ for all $t\in\{t_1, t_2,\dots, t_k\}$, $t'\in\{t_{k+1}, t_{k+2},\dots, t_n\}$ and any reflection occurring in a $T$-reduced expression in the orbit $\mathcal{B}_n\cdot (t_1,t_2,\dots, t_n)$ lies either in $W_u$ or in $W_v$. Since by \cite[Theorem 1.1]{BGRW} there is a unique Hurwitz orbit on $\Red(x)$, the set $\{ t_1, t_2,\dots, t_n\}$ generates $W$. Let $t\in T$. Since a quasi-Coxeter element $x$ has no nontrivial fixed point in $V$ we have that $t\leq_T x$, hence $t$ occurs in a reduced expression of $\Red(x)$. But there is only one orbit $\mathcal{B}_n\cdot (t_1, t_2,\dots, t_k)$ of $\mathcal{B}_n$ on $\Red(w)$. It follows that $t\in W_u$ or in $t\in W_v$. The claimed direct product decomposition follows. 
\end{proof}

\begin{proof}[Proof of Theorem~\ref{mainfinite}]
The existence of the claimed decomposition $w=x_1 x_2\cdots x_m$ is given by Propositions~\ref{main} and \ref{lemquasi}. It remains to show uniqueness. Assume that $w=y_1y_2\cdots y_\ell$ is another such decomposition. For each $i$, choose a reduced expression $t_1^i\cdots t_{n_i}^i$ of $y_i$. Thanks to $(2)$, concatenating these reduced expressions yields a reduced expression of $w$. For a fixed $i$, all the $t_i^j$ are in $P(w)$ by Theorem~\ref{thm:parabreduit}. It follows from Lemma~\ref{lem:parabw} that $t_i^j$ lies in one of the parabolic factor $W(i,j)$ of $P(w)$, and indecomposability of $y_i$ forces $W(i,j)$ to depend only on $i$. Therefore we set $W(i):=W(i,j)$ and note that $y_i\in W(i)$. But since the irreducible parabolic factors of $P(w)$ are precisely the $P(x_j)$'s, for each $W(i)$ there exists $j$ such that $W(i)=P(x_j)$. It follows from the decomposition \begin{equation}\label{eq:1}
P(w)=P(x_1)\times P(x_2)\times\dots P(x_m)
\end{equation}
and the uniqueness of the decomposition of $w$ in the direct product \eqref{eq:1} that each of the $x_j$ can be written as a product of $y_i$'s: more precisely, $x_j$ is the product of all those $y_i$'s such that $W(i)=P(x_j)$ (and for each $j$, there must be at least one $y_i$ such that $W(i)=P(x_j)$ otherwise $w$ would have two distinct decompositions in \eqref{eq:1}). But by indecomposability of $x_j$, there can be at most one such $y_i$, which concludes.      

\end{proof}



\subsection{Consequences}\label{sec:cons}

We conclude with some facts about elements for which
Condition~\ref{con:1} fails. For such an element $w\in W$, by \cite[Theorem 1.1]{BGRW} the Hurwitz operation on $\Red(w)$ is not transitive. Denote by $\mathcal{H}(w)$ the set of orbits. For $\mathcal{O}\in\mathcal{H}(w)$, denote by $\langle \mathcal{O} \rangle$ the reflection subgroup generated by any $T$-reduced expression in $\mathcal{O}$. This is well-defined since the Hurwitz action leaves the subgroup generated by a $T$-reduced expression unchanged. Hence we get the following:

\begin{prop}\label{cor:decn}
Let $(W,S)$ be finite. Let $w\in W$. For each $\mathcal{O}\in\mathcal{H}(w)$, the element $w$ has a unique generalized cycle decomposition in the sense of Theorem~\ref{mainfinite} in the Coxeter group $\langle \mathcal{O} \rangle$. 
\end{prop}

\begin{proof}
The reflection subgroup $\langle\mathcal{O}\rangle$ is a Coxeter group. Denote by $S'$ its set of canonical Coxeter generators. We have (see \cite[Corollary 3.11
(ii)]{Dyer}) that $$T':=\langle\mathcal{O}\rangle\cap T =
\bigcup_{u\in\langle\mathcal{O}\rangle} u S' u^{-1}.$$
Hence $w$ is a quasi-Coxeter element in $(\langle\mathcal{O}\rangle,S')$. Applying Theorem~\ref{mainfinite} to $w$ viewed as element of $\langle\mathcal{O} \rangle$ we get the claim.  

\end{proof}

\begin{lem}\label{orbites}
Let $w\in W$. If $\mathcal{O}_1, \mathcal{O}_2\in\mathcal{H}(w)$ with
$\mathcal{O}_1\neq \mathcal{O}_2$, then $\langle\mathcal{O}_1\rangle\neq
\langle\mathcal{O}_2\rangle.$
\end{lem}
\begin{proof}
Let $(t_1, t_2,\dots, t_k)\in \mathcal{O}_1$. Then $\langle t_1, t_2,
\dots, t_k\rangle=\langle\mathcal{O}_1\rangle.$ Let $S'$ be the set of canonical Coxeter generators of the reflection subgroup
$\langle\mathcal{O}_1\rangle$. The Hurwitz operation is transitive on $\mathrm{Red}_{T_1}(w)$ where $T_1=T\cap \langle\mathcal{O}_1\rangle$ (see the proof of Proposition~\ref{cor:decn}). If we have $\langle \mathcal{O}_1\rangle=\langle \mathcal{O}_2\rangle$, then in particular for $(q_1, q_2,\dots, q_k)\in\mathcal{O}_2$ we have $q_i\in T_1$ for all $i$ since $q_i\in T\cap \langle \mathcal{O}_2\rangle=T_1$. This implies that $(t_1,t_2,\dots,t_k)$ and $(q_1,q_2,\dots,q_k)$ lie in the same Hurwitz orbit since the Hurwitz operation is transitive on $\mathrm{Red}_{T_1}(w)$, a contradiction. 

\end{proof}

\begin{rmq}
Proposition~\ref{cor:decn} tells us that any $w\in W$ has a unique cycle decomposition in any reflection subgroup generated by one of its $T$-reduced expressions. However distinct such reflection subgroups, equivalently (by Lemma~\ref{orbites}) distinct Hurwitz orbits in $\mathcal{H}(w)$ can yield the same decomposition of $w$: more precisely if $x_1, x_2, \dots, x_m$ is a cycle decomposition in $\langle \mathcal{O}_1\rangle$ and $y_1, y_2, \dots, y_\ell$ is a cycle decomposition in $\langle\mathcal{O}_2\rangle$ where $\mathcal{O}_1$ and $\mathcal{O}_2$ are distinct elements in $\mathcal{H}(w)$, then it is possible that $\ell=m$ and $$\{ x_1, x_2,\dots, x_m\}=\{ y_1, y_2,\dots, y_m\}.$$
As an example consider $W$ of type $G_2=I_2(6)$ with $S=\{ s,t\}$. Then $w=stst$ is a Coxeter element in both irreducible reflection subgroups $\langle s, tst\rangle$ and $\langle t, sts\rangle$ of type $A_2$ hence its cycle decomposition is the trivial decomposition $x_1=stst=y_1$ in both subgroups. 

However we always have $$\{ P(x_1)^1, P(x_2)^1,\dots, P(x_m)^1\}\neq\{ P(y_1)^2, P(y_2)^2,\dots, P(y_m)^2\},$$
where $P(x_i)^1$ (resp. $P(y_i)^2$) is the parabolic closure of $x_i$ (resp. $y_i$) in $\langle \mathcal{O}_1 \rangle$ (resp. $\langle \mathcal{O}_2\rangle$). Indeed, the reflections in these subgroups have to generate the parent group $\langle\mathcal{O}_1\rangle$ (resp. $\langle\mathcal{O}_2\rangle$) in which the cycle decomposition is considered and they are distinct by Lemma~\ref{orbites}. In the above example with $W$ of type $G_2$ we have $P(x_1)^1=\langle s, tst\rangle\neq P(y_1)^2=\langle t, sts\rangle$. 

\end{rmq}

Also note the following 

\begin{lem}\label{lem:long}
Let $w\in W$ with $\ell_T(w)=k$. Let $W'\subseteq W$ be a reflection subgroup of $W$ containing $w$ with set of reflections $T'=W'\cap T$. Then $\ell_T(w)=\ell_{T'}(w).$
\end{lem}
\begin{proof}
Since $T'\subseteq T$ we must have $\ell_{T}(w)\leq \ell_{T'}(w)$. Assume that $\ell_{T'}(w)=k'> k$. Let $(t_1, t_2,\dots, t_{k'})\in \mathrm{Red}_{T'}(w)$. Let $i$ be minimal such that $\ell_T(t_1 t_2\cdots t_i)\neq \ell_{T'}(t_1 t_2\cdots t_i)$. Then $\ell_T( t_1 t_2\cdots t_i)= i-2$, $\ell_{T'}(t_1 t_2\cdots t_i)= i$. Let $u= t_1 t_2\cdots t_{i-1}$. By minimality of $i$ we have $\ell_T(u)=\ell_{T'}(u)=i-1$. In particular, we have $t_i\leq_T u$. The parabolic closure $P(u)$ of $u$ therefore has rank $i-1$ and contains $t_1, \dots, t_{i-1}$ but also $t_i$. But since $(t_1, \dots, t_{k'})$ is $T'$-reduced, the reflection subgroup $W''=\langle t_1, t_2,\dots, t_i\rangle$ has rank $i$ (as a reflection subgroup of the Coxeter group $W'$, for instance by Lemma~\ref{carter}). But the reflections of $W''$ as a reflection subgroup of $W'$ or $W$ are the same, hence $W''$ also has rank $i$ as a reflection subgroup of $W$. Therefore $W''$ cannot be included in $P(u)$ which has smaller rank, a contradiction.     
\end{proof}

Hence together with Lemma~\ref{orbites} we get

\begin{cor}\label{cor:corresp}
For all $w\in W$, there is a one-to-one correspondence between $\mathcal{H}(w)$ and reflection subgroups of $W$ in which $w$ is a quasi-Coxeter element, given by $\mathcal{O}\mapsto \langle \mathcal{O} \rangle$. 
\end{cor}




\begin{exple}
Let $W$ be of type $D_4$ with $S=\{ s_0, s_1, s_2, s_3\}$ where $s_2$ commutes with no other simple reflection. Then the element $$w= s_1 (s_2 s_1 s_2) (s_2 s_0 s_2) s_3$$ is a quasi-Coxeter element in $W$ (see~\cite[Example 2.4]{BGRW}), but it is not a Coxeter element (in the sense that it has no $T$-reduced expression yielding a simple system for $W$). It follows from Proposition~\ref{lemquasi} that its cycle decomposition is the trivial decomposition $x_1=w$. Now $W$ can be viewed as a reflection subgroup of a Coxeter group $\widetilde{W}$ of type $B_4$. In that case $W$ is not parabolic in $\widetilde{W}$, hence $w$ has no reduced expression generating a parabolic subgroup of $\widetilde{W}$: indeed, if there was such a decomposition, then the Hurwitz operation on the set of $T$-reduced expressions of $w$ in $\widetilde{W}$ would be transitive by \cite[Theorem 1.1]{BGRW}, hence each reduced expression would generate a parabolic subgroup, in particular $W$ would be parabolic in $\widetilde{W}$. The element $w$ has a unique generalized cycle decomposition in $W$ which is the one we gave above (since $W$ is irreducible), but $w$ can also be realized as a Coxeter element in a (non irreducible) reflection subgroup $W'$ of type $B_2\times B_2$ as follows: in the signed permutation model for the Weyl group $\widetilde{W}$ of type $B_4$ (see \cite[Section 8.1]{BjBr}), we have $w=(1,-2,-1,2)(3,4,-3,-4)$, which is a product of two Coxeter elements of the two type $B_2$ reflection subgroups consisting of those signed permutations supported on $\{\pm 1, \pm 2\}$ and $\{\pm 3, \pm 4\}$ respectively. In $W'$ the unique cycle decomposition of $w$ has two factors $x_1=(1,-2,-1,2)$, $x_2=(3,4,-3,-4)$. Note that neither $x_1$ nor $x_2$ lies in $W$.  
\end{exple}

\section*{Acknowledgments}
I thank Barbara Baumeister and Vivien Ripoll for helpful discussions.



\begin{thebibliography}{10}


\bibitem{Arm} D. ~Armstrong, \emph{Generalized noncrossing partitions and combinatorics of Coxeter groups}, Mem. Amer. Math. Soc. {\bf 202} (2009), No. 949.

\bibitem{BDSW} B.~Baumeister, M.~Dyer, C.~Stump, P.~Wegener, \textsl{A note on the transitive Hurwitz action on decompositions of parabolic Coxeter elements}, \emph{ Proc. Amer. Math. Soc. Ser. B} {\bf 1} (2014), 149-154.
  

\bibitem{BGRW} B.~Baumeister, T.~Gobet, K.~Roberts, P.~Wegener, \textsl{On
the Hurwitz action in finite Coxeter groups},
Journal of Group Theory (2016).

\bibitem{Dual} D.~Bessis, \textsl{The dual braid monoid}, Ann.\
Sci.\ \'Ecole Normale Sup\'erieure {\bf 36} (2003), 647-683.


\bibitem{BjBr} A.~Bj\"orner and F.~Brenti, \textsl{Combinatorics of Coxeter groups}, GTM 231, Springer, 2005.

\bibitem{Bou} N.~Bourbaki, \textsl{Groupes et alg\`ebres de Lie, chapitres
4,5
et 6}, Hermann (1968).

\bibitem{Brady} T.~Brady, \textsl{A partial order on the symmetric group and new $K(\pi,1)$'s for the braid groups}, Adv. Math. {\bf 161} (2001), no. 1, 20-40.



\bibitem{BWnonc} T.~Brady and C.~Watt, \textsl{Non-crossing partition lattices in finite real reflection groups}, Trans. Amer. Math. Soc. {\bf 360} (2008),  no. 4, 1983-2005.






\bibitem{Carter} R.W.~Carter, \textsl{Conjugacy Classes in the Weyl
group}, Compositio Math. {\bf 25} (1972), 1-59.

\bibitem{Dyer} M.J.~Dyer, \emph{Reflection subgroups of Coxeter systems},
J. of Algebra {\bf 135} (1990), Issue 1, 57-73.



\bibitem{Humph} J.~Humphreys, \textsl{Reflection groups and Coxeter groups}, Cambridge Studies in
Advanced Mathematics {\bf 29}, Cambridge University Press (1990).

\bibitem{IT} C.~Ingalls, H.~Thomas, \textsl{Noncrossing partitions and representations of quivers}, Compos. Math. {\bf 145} (2009), no. 6, 1533-1562.


\bibitem{Read} N.~Reading, \textsl{Clusters, Coxeter-sortable elements and noncrossing partitions}, Trans. Amer. Math. Soc. {\bf 359} (2007),  No. 12, 5931-5958.


\end{thebibliography}
\end{document}